\documentclass[a4paper]{amsart}
\usepackage[cp1250]{inputenc}
\usepackage[T1]{fontenc}
\usepackage{amsfonts}
\usepackage{amsmath}
\usepackage{amssymb}
\usepackage{amsthm}
\usepackage{textcomp}
\usepackage{color}

\usepackage{graphicx}
\usepackage[all]{xy}

\usepackage{placeins}
\usepackage{fancyhdr}
\usepackage{enumerate}

\newtheorem{dfn}{Definition}[section]

\newtheorem{prop}[dfn]{Proposition}

\newtheorem{lem}[dfn]{Lemma}

\theoremstyle{definition}
\newtheorem{rem}[dfn]{Remark}

\makeatletter
\renewcommand{\subsection}{\@startsection
	{subsection}
	{2}%
	\z@
	{.5\linespacing\@plus.7\linespacing}
	{.5em}%
	{\normalfont\bfseries}}
\makeatother

\makeatletter
\renewcommand{\subsubsection}{\@startsection
	{subsubsection}
	{2}%
	\z@
	{.5\linespacing\@plus.7\linespacing}
	{.5em}%
	{\normalfont\bfseries}}
\makeatother

\begin{document}
\title{DG quivers of smooth rational surfaces}

\author[A.~Bodzenta]{Agnieszka Bodzenta}
\address{Faculty of Mathematics, Informatics and Mechanics,
         University of Warsaw,
	 Banacha 2,
	 02-097 Warsaw, Poland}
\email{a.bodzenta@mimuw.edu.pl}


\begin{abstract}
Let $X$ be a smooth rational surface. We calculate a DG quiver of a full exceptional collection of line bundles on $X$ obtained by an augmentation from a strong exceptional collection on the minimal model of $X$. In particular, we calculate canonical DG algebras of smooth toric surfaces.
\end{abstract}

\let\thefootnote\relax\footnote{The author is supported by a Polish MNiSW grant (contract number N N201 420639). }

\maketitle 

\section*{Introduction}

Derived categories of coherent sheaves have become one of the main research areas in modern algebraic geometry. An important tool allowing to work with such complicated categories is given by full exceptional collections. Let $X$ be a smooth projective variety and let $D^b(X)$ denote the bounded derived category of coherent sheaves on $X$. It is proved in \cite{bib_B} that a full strong exceptional collection $\sigma$ leads to an equivalence between $D^b(X)$ and the bounded derived category of modules over a finite quiver with relations. By a result of Bondal and Kapranov, see \cite{bib_BK}, if $\sigma$ is not strong then $D^b(X)$ is equivalent to the derived category of modules over some DG category $\mathcal{C}_\sigma$. It is proved in \cite{bib_BS} that in the latter case the DG category $\mathcal{C}_\sigma$ is a path algebra of a finite DG quiver with relations $Q_\sigma$.

Calculating the quiver of a strong exceptional collection is equivalent to understanding endomorphisms of some sheaf. On the other hand, in order to calculate the DG quiver a priori one has to use injective resolutions. In \cite{bib_BS} more comprehensive methods are given for determining DG quivers for two types of exceptional collections. Firstly, if a collection $\sigma$ can be mutated to a strong one $\tau$ then the DG quiver $Q_\sigma$ of $\sigma$ can be calculated by means of the quiver of $\tau$. On the other hand, if a collection $\sigma = \left< \mathcal{E}_1, \ldots, \mathcal{E}_n \right>$ is almost strong, i.e. $\textrm{Ext}^i(\mathcal{E}_j, \mathcal{E}_k) = 0$ for $i \neq 0, 1$, then one can construct a tilting object $\mathcal{E}_\sigma$ using universal extensions and coextensions defined in \cite{bib_HP2}. In this case endomorphisms of $\mathcal{E}_\sigma$ allow to calculate the DG quiver of $\sigma$.

Many examples of almost strong exceptional collections are given by exceptional collections of line bundles on rational surfaces. Recall, that every rational surface $X$, not isomorphic to the projective plane $\mathbb{P}^2$, is obtained from some Hirzebruch surface $\mathbb{F}_a$ by a sequence of blow-ups:
$$
X = X_n\, \xrightarrow{\pi_n}\, X_{n-1} \rightarrow \ldots \rightarrow \, X_1 \, \xrightarrow{\pi_1} \, X_0 = \mathbb{F}_a.
$$ 
In \cite{bib_HP1} Hille and Perling describe an augmentation process that allows to construct full exceptional collections of line bundles on $X$ starting from a full exceptional collection on $X_0$. Moreover, in \cite{bib_HP2} it is proved that collections obtained by augmentation are almost strong.

The main purpose of this note is to calculate the DG quiver of a full exceptional collection $\sigma$ on a smooth rational surface obtained via augmentation from a strong full exceptional collection on $X_0$. To do this we first present $\sigma$ in the canonical form (see Proposition \ref{prop_canonical_form}). Using this presentation we calculate the tilting object $\mathcal{E}_\sigma$ (see Proposition \ref{prop_tilt}), and its endomorphisms. Then using twisted complexes we can calculate the DG quiver of $\sigma$ and any of its mutations.

In Section \ref{sec_toric} we apply these methods to a smooth toric surface $Y$ with $T$-invariant divisors $D_1,\ldots, D_n$. Recall, that divisors $D_i$ correspond to the rays in the fan $\Sigma_Y \subset N \otimes_\mathbb{Z} \mathbb{Q}$ of $Y$. If the order of $D_i$ is induced by an orientation of $\mathbb{Q}^2$ then the collection 
$$
\left< \mathcal{O}_Y, \mathcal{O}_Y(D_1), \mathcal{O}_Y(D_1+D_2), \ldots, \mathcal{O}_Y(D_1+\ldots+D_{n-1})\right>
$$
is full and exceptional on $Y$. For any $k \in \{ 1,\ldots, n\}$ the same remains true for the collection 
$$
\left< \mathcal{O}_Y, \mathcal{O}_Y(D_k), \mathcal{O}_Y(D_k+D_{k+1}), \ldots, \mathcal{O}_Y(D_k+ \ldots + D_{k+n-2}) \right>
$$
if the indices are considered as elements of $\mathbb{Z}/n \mathbb{Z}$. Therefore we can consider all collections of such a form at once. Namely, let $Z = \textrm{Tot}\, \omega_Y$ be the total space of the canonical bundle on $Y$ and let $p\colon Z \rightarrow Y$ denote the canonical projection. As vector bundle $\mathcal{E} = \mathcal{O}_Y \oplus \mathcal{O}_Y(D_1) \oplus \ldots \oplus \mathcal{O}_Y(D_1+\ldots+D_{n-1})$ is a generator of $D^b(Y)$, we know that $p^*(\mathcal{E})$ is a generator of $D^b(Z)$. Moreover,
\begin{align*}
&\textrm{Hom}_Z(p^* (\mathcal{E}), p^*(\mathcal{E})) = \textrm{Hom}_Y(\mathcal{E}, p_* p^*(\mathcal{E})) = \\
&=\textrm{Hom}_Y(\mathcal{E}, \mathcal{E} \otimes p_*(\mathcal{O}_Z)) = \bigoplus_{n\geq 0} \textrm{Hom}_Y(\mathcal{E}, \mathcal{E} \otimes \mathcal{O}_Y(-n K_Y)).
\end{align*}

On $Y$ we can consider an infinite sequence $(A_k)_{k=0}^{\infty}$ of line bundles defined by
$$
A_{k} = \mathcal{O}_Y(s K_Y + D_1 + \ldots+D_r),
$$
where $k = sn + r$ for $0\leq r < n$. Denote by $\mathcal{A}_Y = \oplus A_k$ the sum of all elements in this sequence. By the canonical DG algebra of $Y$ we understand the DG algebra of endomorphisms of $p^*(\mathcal{E})$ or equivalently of $\mathcal{A}_Y$. Methods described in Section \ref{sec_DGquivers} allow us to calculate the canonical DG algebra of any smooth toric surface.

We give examples of canonical DG algebras of smooth toric surfaces. However, we do not investigate the connection between the combinatorial data of the fan $\Sigma_Y$ and the canonical DG algebra of $Y$.

The structure of the paper is as follows. Section \ref{sec_background} contains definitions of quivers and DG quivers, twisted complexes and exceptional collections together with mutations, universal extensions and coextensions. We also recall basic facts about rational surfaces. In Section \ref{sec_DGquivers} we recall after \cite{bib_HP1} the construction of full exceptional collections on smooth rational surfaces. We present any such exceptional collection in the canonical form and describe its Ext-quiver. Then, using universal coextensions, we calculate the associated tilting object and we describe its endomorphisms. This data allows us to calculate the DG quiver of the collection. In Section \ref{sec_toric} we apply these methods to smooth toric surfaces. We start by recalling basic facts about toric surfaces and full exceptional collections on them. Then we define the canonical DG algebra of a toric surface and we show how to use the results of Section \ref{sec_DGquivers} to calculate it. We conclude with examples of the canonical DG algebras for the 1-st and 2-nd Hirzebruch surfaces and for surfaces obtained from the 1-st Hirzebruch surface by blowing-up one point.  

\section{Background}\label{sec_background}

\subsection{Quivers}

A quiver $Q$ consists of two finite sets $Q_0$, $Q_1$ and two maps \mbox{$h, t: Q_1 \rightarrow Q_0$.} Elements of $Q_0$ are vertices of $Q$ and elements of $Q_1$ are arrows of $Q$. The maps $h$ and $t$ indicate the head and the tail of an arrow respectively. A path in $Q$ is a sequence $p = a_n \ldots a_1$ of arrows such that $h(a_i) = t(a_{i+1})$ for $1 \leq i \leq n-1$; we put $h(p) = h(a_n)$ and $t(p) = t(a_1)$. A path algebra $\mathbb{C}\,Q$ of a quiver $Q$ is an algebra with basis consisting of paths in $Q$; the product $p \circ p'$ of two basis elements is defined by means of concatenation of paths if $t(p) = h(p')$ and is put to be zero otherwise. We also assume that for any vertex $i \in Q_0$ there is a trivial path $e_i \in Q_1$ with $t(e_i)= i = h(e_i)$. Then, the element $\sum_{i\in Q_0} e_i $ is the unit of $\mathbb{C}\, Q$.

A quiver with relations $(Q, S)$ is a quiver $Q$ together with a set $S \subset \mathbb{C}\, Q$. Let $I = \left< S\right> \subset \mathbb{C}\,Q$ be an ideal generated by $S$. Then the path algebra $\mathbb{C}\,(Q,S)$ of a quiver with relations is defined to be $\mathbb{C}\,Q / I$. 

If arrows in $Q$ are $\mathbb{Z}$ - graded in such a way that $\textrm{deg}(e_i) = 0$ for any $i \in Q_0$ the path algebra $\mathbb{C}\,Q$ becomes a graded algebra; for a path $p = a_n\ldots a_1$ we put $\textrm{deg}(p) = \textrm{deg}(a_1) +\ldots+\textrm{deg}(a_n)$. 

A DG quiver is a quiver $Q$ together with a $\mathbb{Z}$ grading on $Q_1$ and a structure of a DG algebra on $\mathbb{C} \,Q$ such that $\partial(e_i) = 0$ for any $i \in Q_0$. The Leibniz rule guarantees that $h(\partial(p)) = h(p) $ and $t(\partial(p)) = t(p)$ if only $\partial(p) \neq 0 $. If the set $S \subset \mathbb{C}\, Q$ consists of homogeneous elements one can analogously define a DG quiver with relations $(Q,S)$.

\subsection{Twisted complexes}

Recall that a DG category is a preadditive category $\mathcal{C}$ in which abelian groups $\textrm{Hom}_{\mathcal{C}}(A,B)$ are endowed with a $\mathbb{Z}$-grading and a differential $\partial$ of degree one. Moreover, the composition of morphisms 
$$
\textrm{Hom}_{\mathcal{C}}(A,B)\otimes \textrm{Hom}_{\mathcal{C}}(B,C)\to \textrm{Hom}_{\mathcal{C}}(A,C)
$$
is a morphism of complexes and for any object $C\in\mathcal{C}$ the identity morphism $\textrm{id}_C$ is a closed morphism of degree zero.

For a DG category $\mathcal{C}$ Bondal and Kapranov in \cite{bib_BK} define the category $\mathcal{C}^{\textrm{pre-tr}}$ of twisted complexes. To do this first one has to introduce the category $\widehat{\mathcal{C}}$ of formal shifts. The objects of $\widehat{\mathcal{C}}$ are of the form $C[n]$ where $C\in \mathcal{C}$ and $n\in\mathbb{N}$. For elements $C_1[k]$ and $C_2[n]$ of $\widehat{\mathcal{C}}$ we put $\textrm{Hom}^l_{\widehat{\mathcal{C}}}
(C_1[k], C_2[n]) = \textrm{Hom}^{l+n-k}_\mathcal{C}(C_1, C_2)$ and $\partial_{\widehat{\mathcal{C}}}(f) = (-1)^n \partial_{\mathcal{C}}(f)$ for $f \in \textrm{Hom}_{\widehat{\mathcal{C}}}(C_1[k], C_2[n])$.

A one-sided twisted complex over $\mathcal{C}$ is an expression $(\bigoplus_{i=1}^n C_i[r_i], q_{i,j})$ where $C_i$'s are objects of $\mathcal{C}$, $r_i\in \mathbb{Z}$, $n\geq 0$ and $q_{i,j}\in \textrm{Hom}_{\widehat{\mathcal{C}}}^1(C_i[r_i], C_j[r_j])$ are such that $q_{i,j}=0$ for $i\geq j$ and $\partial q + q^2 = 0$.

Let $C=(\bigoplus C_i[r_i], q)$ and  $C'=(\bigoplus C'_j[r'_j], q')$ be twisted complexes. Morphisms in the category  $\mathcal{C}^\textrm{pre-tr}$ from $C$ to $C'$ are given the set of matrices $f=(f_{i,j})$ for $f_{i,j}\in \textrm{Hom}_{\widehat{\mathcal{C}}}( C_i[r_i], C'_j[r'_j])$. A map $f=(f_{i,j})$ is homogeneous of degree $k$ if \mbox{$f_{i,j}\in \textrm{Hom}^k_{\widehat{\mathcal{C}}}( C_i[r_i], C'_j[r'_j])$} for all pairs $i,j$.

\subsection{Exceptional collections}

Let $X$ be a smooth projective variety defined over $\mathbb{C}$. Recall, that an object $\mathcal{E} \in D^b(X)$ is \emph{exceptional} if $\textrm{Hom}(\mathcal{E},\mathcal{E})= \mathbb{C}$ and $\textrm{Ext}^i(\mathcal{E}, \mathcal{E}) = 0 $ for $i \neq 0$. A sequence $\sigma = \left< \mathcal{E}_1, \ldots,\mathcal{E}_n \right> $ of exceptional sheaves is an \emph{exceptional collection} if $\textrm{Ext}^i(\mathcal{E}_j, \mathcal{E}_k) = 0 $ for $j> k$ and any $i$. An exceptional collection $\sigma$ is \emph{full} if the smallest strictly full subcategory of $D^b(X)$ containing $\mathcal{E}_1, \ldots, \mathcal{E}_n$ equals $D^b(X)$. Finally, the collection $\sigma$ is \emph{strong} if $\textrm{Ext}^i(\mathcal{E}_j, \mathcal{E}_k) =0$ for $i \neq 0 $ and any $j$, $k$.

In \cite{bib_B} it is proved that a full strong exceptional collection $\sigma$ leads to an equivalence of $D^b(X)$ with $D^b(\mathop{\textrm{mod-} {A_\sigma}})$ for a finite dimensional algebra $A_\sigma$. The algebra $\mathcal{A}_\sigma$ is a path algebra of a quiver with relations obtained from sheaves $\mathcal{E}_1, \ldots, \mathcal{E}_n$. 

It is proved in \cite{bib_BK} that when the collection $\sigma$ is not strong the category $D^b(X)$ is equivalent to $D^b(C_\sigma)$ for some DG algebra $C_\sigma$. It was proved in \cite{bib_BS} that $C_\sigma$ can be chosen to be a path algebra of a finite DG quiver with relations $Q_\sigma$. If the collection $\sigma$ is strong the DG algebra $C_\sigma$ is quasi-isomorphic to $A_\sigma$.

In \cite{bib_B} Bondal defines mutations of exceptional collections on a variety $X$. If a pair $\left< \mathcal{E}, \mathcal{F} \right>$ is exceptional then so are the pairs $\left< L_{\mathcal{E}} \mathcal{F}, \mathcal{E}\right>$ and $\left< \mathcal{F}, R_{\mathcal{F}} \mathcal{E}  \right>$ for $L_{\mathcal{E}} \mathcal{F}$ and $R_{\mathcal{F}} \mathcal{E}$ defined by distinguished triangles in $D^b(X)$
\begin{align*}
&L_\mathcal{E} \mathcal{F} \rightarrow \mathcal{E} \otimes \textrm{Hom}(\mathcal{E},\mathcal{F}) \rightarrow \mathcal{F} \rightarrow L_\mathcal{E} \mathcal{F}[1]&\\
&R_\mathcal{F} \mathcal{E}[-1] \rightarrow \mathcal{E} \rightarrow \textrm{Hom}(\mathcal{E},\mathcal{F})^* \otimes \mathcal{F} \rightarrow R_\mathcal{F} \mathcal{E}.&
\end{align*}

For an exceptional collection $\sigma = \langle \mathcal{E}_1,\ldots,\mathcal{E}_n\rangle$ the $i$-th left \emph{mutation} $L_i\sigma$ and the $i$-th right mutation $R_i\sigma$ are exceptional collections defined by 
\begin{align*}
& L_i\sigma = \langle \mathcal{E}_1, \ldots, \mathcal{E}_{i-1}, L_{\mathcal{E}_i}\mathcal{E}_{i+1}, \mathcal{E}_i, \mathcal{E}_{i+2},\ldots, \mathcal{E}_n  \rangle, &\\
& R_i\sigma = \langle \mathcal{E}_1,\ldots, \mathcal{E}_{i-1}, \mathcal{E}_{i+1}, R_{\mathcal{E}_{i+1}}\mathcal{E}_i, \mathcal{E}_{i+2},\ldots, \mathcal{E}_n \rangle.&
\end{align*}

As described in \cite{bib_BS} twisted complexes allow to define mutations of DG quivers in such a way that $Q_{L_i \sigma}  = L_i Q_\sigma$ and $Q_{R_i \sigma} = R_i Q_\sigma$. In particular, it is relatively easy to calculate a DG quiver of a collections that can be mutated to a strong one.

\subsection{Universal extensions and coextensions}

Let $\sigma = \left< \mathcal{E}_1, \ldots, \mathcal{E}_n \right>$ be an exceptional collection. We shall say that $\sigma$ is \emph{almost strong} if \mbox{$\textrm{Ext}^i(\mathcal{E}_j, \mathcal{E}_k) = 0$} for $i \neq 0, \, 1$ and for all $j$, $k$. 

In \cite{bib_HP2} Hille and Perling describe how to construct a tilting object from an almost strong full exceptional collection. The main tool in their construction is universal extension and coextension. For a pair $(\mathcal{E}, \mathcal{F})$ a universal extension $\bar{\mathcal{E}}$ of $\mathcal{E}$ by $\mathcal{F}$ is defined by means of a distinguished triangle
$$
\mathcal{E}[-1] \xrightarrow{can} \mathcal{F} \otimes \textrm{Ext}^1(\mathcal{E}, \mathcal{F})^* \rightarrow \bar{\mathcal{E}} \rightarrow \mathcal{E}.
$$

Dually, a universal coextension of $\mathcal{F}$ by $\mathcal{E}$ is an object $\bar{\mathcal{F}}$ defined by a distinguished triangle
$$
\mathcal{F} \rightarrow \bar{\mathcal{F}} \rightarrow \mathcal{E} \otimes \textrm{Ext}^1(\mathcal{E}, \mathcal{F}) \xrightarrow{\textrm{can}} \mathcal{F}.
$$

If $\mathcal{F}$ is exceptional and $\textrm{Ext}^i(\mathcal{F}, \mathcal{E})= 0$ for all $i$ then $\textrm{Ext}^1(\mathcal{E}, \mathcal{F})^*$ is naturally isomorphic to $\textrm{Hom}(\mathcal{F}, \bar{\mathcal{E}})$ and thus $\mathcal{E}$ is the cone of the canonical map
$$
\mathcal{F}\otimes \textrm{Hom}(\mathcal{F}, \bar{\mathcal{E}}) \xrightarrow{can} \bar{\mathcal{E}} \rightarrow \mathcal{E}.
$$

Dually, if $\mathcal{E}$ is exceptional and $\textrm{Ext}^i(\mathcal{F}, \mathcal{E})= 0$ for all $i$ then $\textrm{Ext}^1(\mathcal{E}, \mathcal{F})$ is naturally isomorphic to $\textrm{Hom}(\bar{\mathcal{F}}, \mathcal{E})$ and up to a shift $\mathcal{F}$ is the cone of the canonical map
$$
\bar{\mathcal{F}} \xrightarrow{\textrm{can}} \mathcal{E} \otimes \textrm{Hom}(\bar{\mathcal{F}}, \mathcal{E}) \rightarrow \mathcal{F}[1].
$$

These observations are used in \cite{bib_BS} to calculate a DG quiver of any almost strong exceptional collection. Again, twisted complexes play an important part in the calculations. 

\subsection{Rational surfaces}

Let $X$ be a smooth rational surface. Then $X$ is obtained by a sequence of blow-ups from the projective plane $\mathbb{P}^2$ or a Hirzebruch surface $\mathbb{F}_a$. We have a sequence of maps 
\[
\xymatrix{X = X_n \ar[r]^{\pi_{n}} & X_{n-1} \ar[r]^{\pi_{n-1}} & \ldots \ar[r]&  X_1 
\ar[r]^{\pi_1} & X_0,} 
\]
where  $X_0 = \mathbb{P}^2$ or $\mathbb{F}_a$. We can also assume that every $\pi_i$ is a blow up of one point $x_{i-1}\in X_{i-1}$.

Let $E_i \subset X_i$ be the exceptional divisor of $\pi_i$. Denote by $E_i$ the strict transform of $E_i$ in $X$ and by $R_i\subset X$ its pullback under $\pi_{i+1} \ldots \pi_n$.

The divisors $R_i$ are mutually orthogonal and $R_i^2 = -1$. Hille and Perling in \cite{bib_HP1} introduce a partial order on the set of indices $\{1,\ldots, n\}$; $i \succeq j$ if $i>j$ and $\pi_{j-1} \ldots \pi_{i-1}(x_{i-1}) = x_{j-1}$. Then 
$$
\textrm{Hom}(\mathcal{O}_X(R_i),\mathcal{O}_X(R_j))= \textrm{Ext}^1(\mathcal{O}_X(R_i), \mathcal{O}_X(R_j)) = \left\{\begin{array}{cl}\mathbb{C} & \textrm{if } i \succeq j , \\0, & \textrm{otherwise.} \end{array} \right.
$$ 
Moreover, $H^0(\mathcal{O}_X(R_i)) = \mathbb{C}$, $H^j(\mathcal{O}_X(R_i)) = 0$ for $j>0$ and $H^k(\mathcal{O}_X(-R_i)) = 0$ for all $k$.

\section{DG quivers of exceptional collections on rational surfaces}\label{sec_DGquivers}

\subsection{Exceptional collections on rational surfaces}

Again, let $X$ be a smooth rational surface. We recall the augmentation procedure given in \cite{bib_HP1} by Hille and Perling which allows to construct full exceptional collections of line bundles on $X$ from an exceptional collection on $X_0$. To simplify the notation we identify a line bundle $\mathcal{L}$ on $X_i$ with its pull back via $\pi_j$'s and denote them by the same letter.

Let $\sigma=\langle\mathcal{L}_1,\ldots,\mathcal{L}_s\rangle$ be a full exceptional collection of line bundles on $X_i$. The augmentation of $\sigma$ is \mbox{$\sigma '=\langle \mathcal{L}_1(R_{i+1}), \ldots, \mathcal{L}_{k-1}(R_{i+1}), \mathcal{L}_k,$ $  \mathcal{L}_k(R_{i+1}), \mathcal{L}_{k+1}, \ldots \mathcal{L}_s \rangle$} -- an exceptional collection on $X_{i+1}$. 

It follows from a result of Orlov, \cite{bib_O} that collections obtained via augmentation are full. It is proved in \cite{bib_HP2} that they are almost strong.

Mutations allow to present each of the above described collections in the following form.
 
\begin{prop}\label{prop_canonical_form}
Any exceptional collection of line bundles on $X$ obtained via augmentation can be mutated to $\langle \mathcal{O}_{R_n}(R_n)[-1], \ldots, \mathcal{O}_{R_1}(R_1)[-1],\mathcal{O}_X, \mathcal{N}_1, \ldots, \mathcal{N}_t \rangle $, where $\langle \mathcal{O}_{X_0} = \mathcal{N}_0, \mathcal{N}_1, \ldots, \mathcal{N}_t \rangle$  is an exceptional collection on $X_0$.
\end{prop}

\begin{proof}

The collection on $X$ obtained via augmentation is of the form 
$$
\langle \mathcal{L}_1(R_n),\ldots,\mathcal{L}_{i-1}(R_n),\mathcal{L}_i, \mathcal{L}_i(R_n), \mathcal{L}_{i+1}, \ldots \mathcal{L}_s \rangle,
$$
where $\mathcal{L}_j$'s are pull backs of line bundles on $X_{n-1}$.  

Equality
$$
 \textrm{Hom}(\mathcal{L}_i, \mathcal{L}_i(R_n)) = \textrm{Hom}(\mathcal{O}_X,\mathcal{O}_X(R_n)) = \mathbb{C}
$$
and the short exact sequence
$$
0\rightarrow \mathcal{L}_i \rightarrow \mathcal{L}_i(R_t) \rightarrow \mathcal{O}_{R_i}(R_i) \rightarrow 0
$$ 
show that this collection can be mutated to 
$$
\langle \mathcal{L}_1(R_n), \ldots, \mathcal{L}_{i-1}(R_n), \mathcal{O}_{R_n}(R_n)[-1], \mathcal{L}_i, \mathcal{L}_{i+1}, \ldots \mathcal{L}_s \rangle.
$$

Then 
$$
\textrm{Hom}(\mathcal{L}_i(R_n), \mathcal{O}_{R_n}(R_n)) = \textrm{Hom}(\mathcal{O}_X(R_n), \mathcal{O}_{R_n}(R_n)) = \textrm{Hom}(\mathcal{O}_X, \mathcal{O}_{R_n}) = \mathbb{C}
$$
and exact sequences
$$
0\rightarrow \mathcal{L}_i \rightarrow \mathcal{L}_i(R_n) \rightarrow \mathcal{O}_{R_n}(R_n) \rightarrow 0
$$
provide further mutations to $\langle \mathcal{O}_{R_n}(R_n)[-1], \mathcal{L}_1, \ldots, \mathcal{L}_s \rangle$. 

The collection $\langle \mathcal{L}_1,\ldots, \mathcal{L}_s \rangle$ is a pull back of a collection on $X_{n-1}$ and it again has the form $\langle \mathcal{L}'_1 (R_{n-1}), \ldots, \mathcal{L}'_{k-1}(R_{n-1}), \mathcal{L}'_k, \mathcal{L}'_k(R_{n-1}), \mathcal{L}'_{k+1}, \ldots, \mathcal{L}'_{s-1} \rangle$ for some $k$. As before, it can be mutated to $\langle \mathcal{L}'_1(R_{n-1}), \ldots, \mathcal{L}'_{k-1}(R_{n-1}),$ $ \mathcal{O}_{R_{n-1}}(R_{n-1})[-1],\mathcal{L}'_k, \ldots, \mathcal{L}'_{s-1}\rangle$ and then to $\langle \mathcal{O}_{R_{n-1}}(R_{n-1})[-1],\mathcal{L}'_1, \ldots, \mathcal{L}'_{s-1} \rangle$.

Continuing, we can mutate the collection on $X$ to $\langle \mathcal{O}_{R_n}(R_n)[-1], \ldots, \mathcal{O}_{R_1}(R_1)[-1],$ $ \mathcal{O}_X, \mathcal{N}_1, \ldots, \mathcal{N}_t \rangle$. 
\end{proof}

From now on we will assume that the collection $\left< \mathcal{O}_{X_0}, \mathcal{N}_1,\ldots, \mathcal{N}_t \right>$ on $X_0$ is strong.

\subsection{Ext-quiver of $\langle \mathcal{O}_{R_n}(R_n)[-1], \ldots, \mathcal{O}_{R_1}(R_1)[-1], \mathcal{O}_X, \mathcal{N}_1, \ldots, \mathcal{N}_t \rangle$.}

To draw the Ext-quiver of this collection in particular we need to understand the compositions 
\begin{align*}
\textrm{Ext}^1(\mathcal{O}_{R_j}(R_j), \mathcal{O}_{R_k}(R_k)) \otimes \textrm{Hom}(\mathcal{O}_{R_i}(R_i), \mathcal{O}_{R_j}(R_j)) &\rightarrow \textrm{Ext}^1(\mathcal{O}_{R_i}(R_i), \mathcal{O}_{R_k}(R_k)),\\
\textrm{Hom}(\mathcal{O}_{R_j}(R_j), \mathcal{O}_{R_k}(R_k)) \otimes \textrm{Ext}^1(\mathcal{O}_{R_i}(R_i), \mathcal{O}_{R_j}(R_j)) &\rightarrow \textrm{Ext}^1(\mathcal{O}_{R_i}(R_i), \mathcal{O}_{R_k}(R_k))
\end{align*}
for $i \succeq j \succeq k$. 

Denote by $\mathcal{C}$ the subcategory of $D^b(X)$ generated by objects $\mathcal{O}_{R_n}(R_n), \ldots, \mathcal{O}_{R_1}(R_1)$ and by $\mathcal{C}'$ the subcategory of $D^b(X)$ generated by $\mathcal{O}(R_n), \ldots, \mathcal{O}(R_1)$. Then $\mathcal{C}$ is a mutation of $\mathcal{C}'$ over $\mathcal{O}_X$ and hence understanding morphisms between generators of $\mathcal{C}$ is equivalent to understanding morphisms between generators of $\mathcal{C}'$.

\begin{lem}\label{lem_composition}
Let $i \succeq j \succeq k$. The composition 
$$
\textrm{Ext}^1(\mathcal{O}_X(R_j), \mathcal{O}_X(R_k)) \otimes \textrm{Hom}(\mathcal{O}_X(R_i), \mathcal{O}_X(R_j)) \rightarrow \textrm{Ext}^1(\mathcal{O}_X(R_i), \mathcal{O}_X(R_k))
$$
is an isomorphism.
\end{lem}
\begin{proof}
The exact sequence
$$
0 \rightarrow \mathcal{O}_X(R_i) \rightarrow \mathcal{O}_X(R_j) \rightarrow \mathcal{O}_{R_j-R_i}(R_j) \rightarrow 0 
$$
gives
\begin{align*}
0 \rightarrow &\textrm{Hom}(\mathcal{O}_{R_j-R_i}(R_j), \mathcal{O}_X(R_k)) \rightarrow \textrm{Hom}(\mathcal{O}_X(R_j), \mathcal{O}_X(R_k)) \xrightarrow{\alpha} \textrm{Hom}(\mathcal{O}_X(R_i), \mathcal{O}_X(R_k))\rightarrow \\
\rightarrow &\textrm{Ext}^1(\mathcal{O}_{R_j-R_i}(R_j), \mathcal{O}_X(R_k)) \rightarrow \textrm{Ext}^1(\mathcal{O}_X(R_j), \mathcal{O}_X(R_k)) \xrightarrow{\beta} \textrm{Ext}^1(\mathcal{O}_X(R_i), \mathcal{O}_X(R_k))\rightarrow\\
\rightarrow &\textrm{Ext}^2(\mathcal{O}_{R_j-R_i}(R_j), \mathcal{O}_X(R_k)) \rightarrow 0.
\end{align*}
The morphism $\alpha \colon \textrm{Hom}(\mathcal{O}_X(R_j), \mathcal{O}_X(R_k)) \rightarrow \textrm{Hom}(\mathcal{O}_X(R_i), \mathcal{O}_X(R_k))$ is an isomorphism because its kernel is zero and both spaces are one dimensional.

$\beta$ is an isomorphism if and only if $\textrm{Ext}^1(\mathcal{O}_{R_j-R_i}(R_j), \mathcal{O}_X(R_k))$ is zero.

We have a short exact sequence
$$
0 \rightarrow \mathcal{O}_{R_j-R_i}(R_j) \rightarrow \mathcal{O}_{R_j}(R_j+R_i) \rightarrow \mathcal{O}_{R_i}(R_j + R_i)\simeq \mathcal{O}_{R_i}(R_i) \rightarrow 0.
$$
It is easy to check that $\textrm{Ext}^1(\mathcal{O}_{R_i}(R_i), \mathcal{O}_{R_k}(R_k)) = \mathbb{C}$. From short exact sequences
\begin{align*}
&0 \rightarrow \mathcal{O}_X(R_i) \rightarrow \mathcal{O}_X(R_i+R_j) \rightarrow \mathcal{O}_{R_j}(R_i+R_j) \rightarrow 0, &\\ 
&0 \rightarrow \mathcal{O}_X(R_j) \rightarrow \mathcal{O}_X(R_i+R_j) \rightarrow \mathcal{O}_{R_i}(R_i+R_j) \simeq \mathcal{O}_{R_i}(R_i) \rightarrow 0 & 
\end{align*}
we deduce that 
$$
\textrm{Ext}^1(\mathcal{O}_{R_j}(R_j+R_i), \mathcal{O}_X(R_k))  \simeq \textrm{Ext}^1(\mathcal{O}_{R_j}(R_j+R_i), \mathcal{O}_X)  \simeq \textrm{Ext}^1(\mathcal{O}_X(R_j+R_i), \mathcal{O}_X) = \mathbb{C}.
$$
It follows that $\textrm{Ext}^1(\mathcal{O}_{R_j-R_i}(R_j), \mathcal{O}_X(R_k)) = 0$.
\end{proof}

\begin{rem}
If $i \succeq j \succeq k$ the composition 
$$
\textrm{Hom}(\mathcal{O}_X(R_j), \mathcal{O}_X(R_k)) \otimes \textrm{Ext}^1(\mathcal{O}_X(R_i), \mathcal{O}_X(R_j)) \rightarrow \textrm{Ext}^1(\mathcal{O}_X(R_i), \mathcal{O}_X(R_k))
$$
does not have to be an isomorphism. Indeed, consider a surface $X$ obtained from its minimal model by three blow-ups such that $E_1^2 = -3$, $E_2^2 = -2$, $E_3^2 = -1$, $E_1 E_2= 0$, $E_1 E_3 = 1 $ and $E_2 E_3= 1$. Then $R_1= E_1+E_2+2E_3$, $R_2 = E_2+E_3$ and $R_3 = E_3$. Let $\bar{\alpha} \in \textrm{Ext}^1(\mathcal{O}_X(R_3), \mathcal{O}_X(R_2))$ and $\beta \in \textrm{Hom}(\mathcal{O}_X(R_2), \mathcal{O}_X(R_1))$ be non-zero elements. We have a short exact sequence
$$
0 \rightarrow \mathcal{O}_X(E_2+E_3) \xrightarrow{\beta} \mathcal{O}_X(E_1+E_2+2E_3) \rightarrow \mathcal{O}_{E_1+E_3}(E_1+E_2+2E_3) \rightarrow 0.
$$
As in the proof of the previous lemma $\beta \circ \bar{\alpha} = 0 $ if and only if $\textrm{Hom}(\mathcal{O}_X(E_3), \mathcal{O}_{E_1+E_3}(E_1+E_2+2E_3)) \neq 0$. We have $\textrm{Hom}(\mathcal{O}_X(E_3), \mathcal{O}_{E_1+E_3}(E_1+E_2+2E_3)) = H^0(X, \mathcal{O}_{E_1+E_3}(E_1+E_2+E_3))$. The latter sheaf fits into a short exact sequence
$$
0 \rightarrow \mathcal{O}_{E_3} \simeq \mathcal{O}_{E_3}(E_2+E_3) \rightarrow \mathcal{O}_{E_1+E_3}(E_1+E_2+E_3) \rightarrow \mathcal{O}_{E_1}(E_1+E_2+E_3) \simeq \mathcal{O}_{E_1}(-2) \rightarrow 0 
$$
from which it follows that $ H^0(X, \mathcal{O}_{E_1+E_3}(E_1+E_2+E_3)) = \mathbb{C}$.
\end{rem}

Thus, we know that between $\mathcal{O}_{R_i}(R_i)$ and $\mathcal{O}_{R_j}(R_j)$ there is either no arrow or two arrows, one in degree zero and one in degree one. Moreover, $\bar{\beta} \circ \alpha \neq 0 $ and $\beta \circ \alpha \neq 0$ for
\[
\xymatrix{\mathcal{O}_{R_i}(R_i) \ar@<1ex>[r]^{\alpha} \ar@<-1ex>[r]_{\bar{\alpha}} & \mathcal{O}_{R_k}(R_k) \ar@<1ex>[r]^{\beta} \ar@<-1ex>[r]_{\bar{\beta}} & \mathcal{O}_{R_j}(R_j),}
\]
where $\alpha, \beta$ are non-zero morphisms and $\bar{\alpha}, \bar{\beta}$ are non-zero elements of the first Ext groups.

It remains to understand what are the maps from $\mathcal{O}_{R_k}(R_k)$ to $\mathcal{N}_i$ and the compositions between them.

As $\mathcal{N}_i$ are torsion-free we know that \mbox{$\textrm{Hom}(\mathcal{O}_{R_k}(R_k), \mathcal{N}_i) = 0$.} From the short exact sequence
\begin{equation}\label{eqn1}
0 \rightarrow \mathcal{N}_i \rightarrow \mathcal{N}_i \otimes \mathcal{O}_X(R_k) \rightarrow \mathcal{O}_{R_k}(R_k)\rightarrow 0.
\end{equation}
we deduce that $\textrm{Ext}^1(\mathcal{O}_{R_k}(R_k), \mathcal{N}_i) \simeq \textrm{Hom}(\mathcal{O}_{R_k}(R_k), \mathcal{O}_{R_k}(R_k))  = \mathbb{C}$. Let $\zeta_k^i$ denote the non-zero element of the group $\textrm{Ext}^1(\mathcal{O}_{R_k}(R_k), \mathcal{N}_i)$.

Moreover, the diagram
\[
\xymatrix{& & 0 \ar[d] & 0 \ar[d] & \\
0 \ar[r] & \mathcal{N}_i \ar[r] \ar[d]^= & \mathcal{N}_i\otimes \mathcal{O}_X(R_j) \ar[r]  \ar[d] & \mathcal{O}_{R_j}(R_j) \ar[d] \ar[r] & 0  \\
0 \ar[r] & \mathcal{N}_i \ar[r] & \mathcal{N}_i \otimes \mathcal{O}_X(R_k) \ar[r] \ar[d] & \mathcal{O}_{R_k}(R_k) \ar[r] \ar[d] & 0 \\
 & & \mathcal{O}_{R_k-R_j}(R_k) \ar[r]^= \ar[d] & \mathcal{O}_{R_k-R_j}(R_k) \ar[d]\\
& & 0 & 0 } 
\]
shows that the composition
$$
\textrm{Hom}(\mathcal{O}_{R_j}(R_j), \mathcal{O}_{R_k}(R_k)) \otimes \textrm{Ext}^1(\mathcal{O}_{R_k}(R_k), \mathcal{N}_i) \rightarrow \textrm{Ext}^1(\mathcal{O}_{R_j}(R_j), \mathcal{N}_i)
$$
is an isomorphism.

To understand the composition 
$$
\textrm{Ext}^1(\mathcal{O}_{R_k}(R_k), \mathcal{N}_i) \otimes \textrm{Hom}(\mathcal{N}_i,\mathcal{N}_l) \rightarrow \textrm{Ext}^1(\mathcal{O}_{R_k}(R_k), \mathcal{N}_l)
$$
we apply the functor $\textrm{Hom}(-,\mathcal{N}_l)$ to the short exact sequence (\ref{eqn1}). It follows that for $\phi \in \textrm{Hom}(\mathcal{N}_i, \mathcal{N}_l)$ the composition $\phi \circ \zeta_k^i$ is zero if and only if $\phi$ factors through $\mathcal{N}_l(-R_k)$.

\subsection{DG quiver of $\langle \mathcal{O}_{R_n}(R_n)[-1], \ldots, \mathcal{O}_{R_1}(R_1)[-1], \mathcal{O}_X, \mathcal{N}_1, \ldots, \mathcal{N}_t \rangle$}

Now, we will present calculations allowing to determine the DG quiver of the collection $\langle \mathcal{O}_{R_n}(R_n)[-1], \ldots, \mathcal{O}_{R_1}(R_1)[-1], \mathcal{O}, \mathcal{N}_1, \ldots, \mathcal{N}_t \rangle$. Recall, that we work under the assumption that the collection $\left<\mathcal{O}_{X_0}, \mathcal{N}_1, \ldots, \mathcal{N}_t \right>$ on $X_0$ is strong.

To calculate the DG category of the collection $\langle \mathcal{O}_{R_n}(R_n)[-1], \ldots, \mathcal{O}_{R_1}(R_1)[-1], \mathcal{O}_X,$ $ \mathcal{N}_1, \ldots, \mathcal{N}_t \rangle$ we substitute some objects with universal coextensions.

\subsubsection{Tilting object}\label{sssec:tilt}

Note that if $2 \succeq 1$ then we have a unique nontrivial extension
$$
0 \rightarrow \mathcal{O}_{R_1}(R_1) \rightarrow \mathcal{O}_{R_1+R_2}(R_1 + R_2) \rightarrow \mathcal{O}_{R_2}(R_2) \rightarrow 0.
$$
Hence $\mathcal{O}_{R_1+R_2}(R_1+R_2)$ is the universal coextension of $\mathcal{O}_{R_1}(R_1)$ by $\mathcal{O}_{R_2}(R_2)$.

We will show that for $i_k \succeq \ldots \succeq i_1 \succeq s $ the universal coextension of $\mathcal{O}_{R_{i_1}+ \ldots + R_{i_k}}(R_{i_1} + \ldots + R_{i_k})$ by $\mathcal{O}_{R_s}(R_s)$ is $\mathcal{O}_{R_s + R_{i_1} + \ldots + R_{i_k}}(R_s+R_{i_1} + \ldots + R_{i_k})$. 

\begin{prop}\label{prop_tilt}
Let $\left< \mathcal{O}_{R_n}(R_n)[-1], \ldots, \mathcal{O}_{R_1}(R_1)[-1], \mathcal{O}_X, \mathcal{N}_1, \ldots, \mathcal{N}_t \right>$ be an exceptional collection on $X$ such that $\left< \mathcal{O}_{X_0}, \mathcal{N}_1, \ldots, \mathcal{N}_t\right>$ is a strong exceptional collection on $X_0$. Then 
$$
 \mathcal{O}_{S_n}(S_n)[-1]\, \oplus \, \mathcal{O}_{S_{n-1}}(S_{n-1})[-1]\oplus \ldots \oplus \mathcal{O}_{S_1}(S_1)[-1] \, \oplus \, \mathcal{O}_X\, \oplus \, \mathcal{N}_1 \oplus \ldots \oplus \mathcal{N}_t 
$$
is tilting on $X$, where $S_k$ are defined as
$$
S_k = \sum_{ j\succeq k} R_j.
$$
\end{prop}

To prove Proposition \ref{prop_tilt} we shall need the following Lemma.

\begin{lem}\label{lem_hominto}
For $i \succeq k \succeq l$ we have
\begin{align*}
&\textrm{Hom}(\mathcal{O}_{R_i}(R_i), \mathcal{O}_{R_l + \ldots+R_k}(R_l + \ldots+R_k)) \simeq\\ &\textrm{Hom}(\mathcal{O}_{R_i}(R_i), \mathcal{O}_{R_{k+1}}(R_{k+1})) \otimes \textrm{Hom}(\mathcal{O}_{R_{k+1}}(R_{k+1}), \mathcal{O}_{R_l+ \ldots+R_k}(R_l + \ldots+R_k)) = \mathbb{C},&  \\
&\textrm{Hom}(\mathcal{O}_{R_i}(R_i), \mathcal{O}_{R_l+R_{l+1} + \ldots+R_k}(R_l+R_{l+1} + \ldots+R_k)) \simeq \\ &\textrm{Ext}^1(\mathcal{O}_{R_i}(R_i), \mathcal{O}_{R_{k+1}}(R_{k+1})) \otimes \textrm{Ext}^1(\mathcal{O}_{R_{k+1}}(R_{k+1}), \mathcal{O}_{R_l+ \ldots+R_k}(R_l + \ldots+R_k)) = \mathbb{C},&
\end{align*}
where the sum $R_l + \ldots +R_k$ is taken over all divisors $R_j$ such that $k \succeq j \succeq l$.
\end{lem}
\begin{proof}
We proceed by induction. The basis case, for $k =l$ follows from Lemma \ref{lem_composition}. The induction step follows from applying the functor $\textrm{Hom}(\mathcal{O}_{R_i}(R_i), -)$ to the short exact sequence 
\begin{align}\label{eqt_extension}
0 \rightarrow \mathcal{O}_{R_l+\ldots+R_{k-1}}(R_l+\ldots+R_{k-1}) \rightarrow \mathcal{O}_{R_l+\ldots+R_{k}}(R_l+\ldots+R_{k}) \rightarrow \mathcal{O}_{R_k}(R_k) \rightarrow 0.
\end{align}
\end{proof}

\begin{proof}[Proof of Proposition \ref{prop_tilt}] 
From the above lemma and the short exact sequence (\ref{eqt_extension}) it follows that if $i \succeq k \succeq l$ the sheaf  $\mathcal{O}_{R_l+\ldots+R_{k}}(R_l+\ldots+R_{k})$ is the universal coextension of $\mathcal{O}_{R_l+\ldots+R_{k-1}}(R_l+\ldots+R_{k-1})$ by $\mathcal{O}_{R_i}(R_i)$. Hence, by the construction described in \cite{bib_HP2} 
the object
$$
 \mathcal{O}_{S_n}(S_n)[-1]\, \oplus \, \mathcal{O}_{S_{n-1}}(S_{n-1})[-1]\oplus \ldots \oplus \mathcal{O}_{S_1}(S_1)[-1] \, \oplus \, \mathcal{O}_X\, \oplus \, \mathcal{N}_1 \oplus \ldots \oplus \mathcal{N}_t 
$$
is tilting on $X$.
\end{proof}

Endomorphisms of the tilting object depend not only on the order of the divisors but also on the mutual position of the exceptional curves.

Consider the blow-up with the following mutual position of exceptional divisors.
\begin{center}

\begingroup
  \makeatletter
  \providecommand\color[2][]{%
    \errmessage{(Inkscape) Color is used for the text in Inkscape, but the package 'color.sty' is not loaded}
    \renewcommand\color[2][]{}%
  }
  \providecommand\transparent[1]{%
    \errmessage{(Inkscape) Transparency is used (non-zero) for the text in Inkscape, but the package 'transparent.sty' is not loaded}
    \renewcommand\transparent[1]{}%
  }
  \providecommand\rotatebox[2]{#2}
  \ifx\svgwidth\undefined
    \setlength{\unitlength}{240pt}
  \else
    \setlength{\unitlength}{\svgwidth}
  \fi
  \global\let\svgwidth\undefined
  \makeatother
  \begin{picture}(1,0.26666667)%
    \put(0,0){\includegraphics[width=\unitlength]{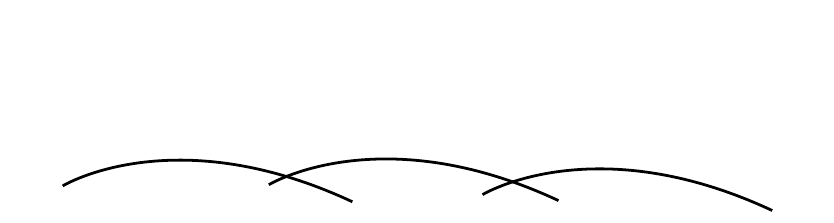}}%
    \put(0.73,0.08855256){\color[rgb]{0,0,0}\makebox(0,0)[lb]{\smash{$E_1$}}}%
    \put(0.18714284,0.09712373){\color[rgb]{0,0,0}\makebox(0,0)[lb]{\smash{$E_3$}}}%
    \put(0.44428574,0.09474275){\color[rgb]{0,0,0}\makebox(0,0)[lb]{\smash{$E_2$}}}%
  \end{picture}%
\endgroup

\end{center}
Then 
\begin{align*}
&E_1^2 = -2,& &E_2^2 = -2,& &E_3^2 = -1,&\\
&E_1E_2=1,& & E_1E_3 = 0,& & E_2E_3 = 1&
&R_1=E_1+E_2+E_3,& & R_2=E_2+E_3,&  &R_3 = E_3,&\\
\end{align*}
and the order is
\begin{align*}
& & &3 \succeq 2 \succeq 1.& & &
\end{align*}
The endomorphisms of $\mathcal{O}_{R_3}(R_3) \oplus \mathcal{O}_{R_2+R_3}(R_2+R_3) \oplus \mathcal{O}_{R_1+R_2+R_3}(R_1+R_2+R_3)$ are
\[
\xymatrix{\mathcal{O}_{R_3}(R_3) \ar@<1ex>[r]^{\alpha_3} & \mathcal{O}_{R_2+ R_3}(R_2+R_3)  \ar@<1ex>[r]^{\alpha_2}  \ar@<1ex>[l]^{\beta_3} &  \mathcal{O}_{R_1+R_2+ R_3}(R_1+R_2+R_3) \ar@<1ex>[l]^{\beta_2}}
\]
with
\begin{align*}
&\beta_3\circ \alpha_3 = 0,& & \alpha_3 \circ \beta_3 = \beta_2 \circ \alpha_2.&
\end{align*}

However, if the picture is
\begin{center}

\begingroup
  \makeatletter
  \providecommand\color[2][]{%
    \errmessage{(Inkscape) Color is used for the text in Inkscape, but the package 'color.sty' is not loaded}
    \renewcommand\color[2][]{}%
  }
  \providecommand\transparent[1]{%
    \errmessage{(Inkscape) Transparency is used (non-zero) for the text in Inkscape, but the package 'transparent.sty' is not loaded}
    \renewcommand\transparent[1]{}%
  }
  \providecommand\rotatebox[2]{#2}
  \ifx\svgwidth\undefined
    \setlength{\unitlength}{240pt}
  \else
    \setlength{\unitlength}{\svgwidth}
  \fi
  \global\let\svgwidth\undefined
  \makeatother
  \begin{picture}(1,0.26666667)%
    \put(0,0){\includegraphics[width=\unitlength]{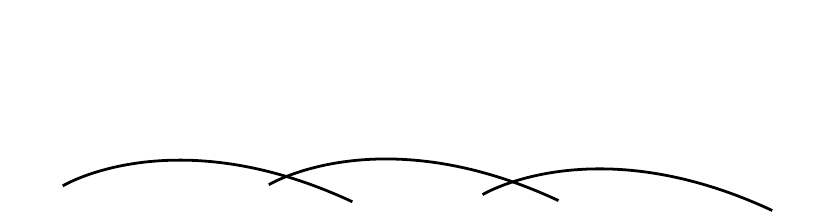}}%
    \put(0.73,0.08855256){\color[rgb]{0,0,0}\makebox(0,0)[lb]{\smash{$E_1$}}}%
    \put(0.18714284,0.09712373){\color[rgb]{0,0,0}\makebox(0,0)[lb]{\smash{$E_2$}}}%
    \put(0.44428574,0.09474275){\color[rgb]{0,0,0}\makebox(0,0)[lb]{\smash{$E_3$}}}%
  \end{picture}%
\endgroup

\end{center}
then
\begin{align*}
&E_1^2 = -3,& &E_2^2 = -2,& &E_3^2 = -1,&\\
&E_1E_2=0,& & E_1E_3 = 1,& & E_2E_3 = 1,&\\
&R_1=E_1+E_2+2E_3,& & R_2=E_2+E_3,&  &R_3 = E_3,&\\
\end{align*}
the order is still
\begin{align*}
& & &3 \succeq 2 \succeq 1.& & &
\end{align*}
and the endomorphisms of the tilting object are
\[
\xymatrix{\mathcal{O}_{R_3}(R_3) \ar@<1ex>[r]^(0.4){\alpha_3} & \mathcal{O}_{R_2+ R_3}(R_2+R_3)  \ar@<1ex>[r]^(0.4){\alpha_2}  \ar@<1ex>[l]^(0.6){\beta_3} &  \mathcal{O}_{R_1+R_2+ R_3}(R_1+R_2+R_3) \ar@<1ex>[l]^(0.6){\beta_2}}
\]
with
\begin{align*}
&\beta_3\circ \alpha_3 = 0,& &  \beta_2 \circ \alpha_2 = 0.&
\end{align*}

\subsubsection{Ext$^1(\mathcal{O}_{S_k}(S_k), \mathcal{N}_i)$}

\begin{lem}
Let $i_k \succeq i_{k-1} \succeq \ldots \succeq i_1$. Then 
$$
\textrm{Ext}^1(\mathcal{O}_{R_{i_1}+\ldots + R_{i_k}}(R_{i_1}+\ldots + R_{i_k}), \mathcal{N}_i) = \mathbb{C}^k
$$
and the remaining Ext groups are zero.
\end{lem}
\begin{proof}
We proceed by induction. The short exact sequence
$$
0 \rightarrow \mathcal{O}_{R_{i_1}+\ldots+R_{i_{k-1}}}(R_{i_1}+\ldots+R_{i_{k-1}}) \rightarrow \mathcal{O}_{R_{i_1}+\ldots+R_i{_{k}}}(R_{i_1}+\ldots+R_{i_k}) \rightarrow \mathcal{O}_{R_{i_k}}(R_{i_k}) \rightarrow 0
$$
together with an equality 
$$
\textrm{Ext}^i(\mathcal{O}_{R_{i_k}}(R_{i_k}), \mathcal{N}_i) = \textrm{Ext}^i(\mathcal{O}_{E_{i_k}}(E_{i_k}), \mathcal{N}_i)
$$
completes the proof.
\end{proof}

If we apply the functor $\textrm{Hom}(\mathcal{O}_{S_k}(S_k),-)$ to the short exact sequence
$$
0 \rightarrow \mathcal{N}_i \rightarrow \mathcal{N}_i\otimes \mathcal{O}_X(S_k) \rightarrow \mathcal{O}_{S_k}(S_k) \rightarrow 0
$$
we get an isomorphism
\begin{align}
\label{eqn-Hom-and-Ext}
\textrm{Ext}^1(\mathcal{O}_{S_k}(S_k), \mathcal{N}_i) \simeq \textrm{Hom}(\mathcal{O}_{S_k}(S_k), \mathcal{O}_{S_k}(S_k)).
\end{align}
The identity morphism in the latter space corresponds to an element \mbox{$\zeta_k^i \in \textrm{Ext}^1(\mathcal{O}_{S_k}(S_k), \mathcal{N}_i)$.} 

The diagram
\[
\xymatrix{& & 0 \ar[d] & 0 \ar[d] &\\
0 \ar[r] &\mathcal{N}_i \ar[r] \ar[d]^= & \mathcal{N}_i \otimes \mathcal{O}_X(S_k) \ar[d] \ar[r] & \mathcal{O}_{S_k}(S_k)\ar[d]^{\iota} \ar[r] & 0 \\
0 \ar[r] & \mathcal{N}_i \ar[r] & \mathcal{N}_i\otimes \mathcal{O}_X(S_l) \ar[r] \ar[d] & \mathcal{O}_{S_l}(S_l) \ar[d]\ar[r]& 0\\
& & \mathcal{O}_{S_l-S_k}(S_l) \ar[r]^= \ar[d] & \mathcal{O}_{S_l-S_k}(S_l) \ar[d] &\\
& & 0 & 0 &}
\]
shows that for an inclusion $\iota\colon \mathcal{O}_{S_k}(S_k) \rightarrow \mathcal{O}_{S_l}(S_l)$ we have $\zeta_l^i \circ \iota = \zeta_k^i$.

The isomorphism (\ref{eqn-Hom-and-Ext}) allows also to calculate the Yoneda composition 
$$
\textrm{Hom}(\mathcal{N}_i, \mathcal{N}_k) \otimes \textrm{Ext}^1(\mathcal{O}_{S_k}(S_k), \mathcal{N}_i) \rightarrow \textrm{Ext}^1(\mathcal{O}_{S_k}(S_k), \mathcal{N}_k).
$$

Thus, if the collection $\left< \mathcal{O}_{X_0}, \mathcal{N}_1, \ldots, \mathcal{N}_t \right>$ on $X_0$ is strong we know the endomorphism algebra of the tilting object 
$$
 \mathcal{O}_{S_n}(S_n)[-1]\, \oplus \, \mathcal{O}_{S_{n-1}}(S_{n-1})[-1]\oplus \ldots \oplus \mathcal{O}_{S_1}(S_1)[-1] \, \oplus \, \mathcal{O}_X\, \oplus \, \mathcal{N}_1 \oplus \ldots \oplus \mathcal{N}_t. 
$$
Using twisted complexes one can then calculate the DG quiver of the collection $\left< \mathcal{O}_{R_n}(R_n)[-1], \ldots, \mathcal{O}_{R_1}(R_1)[-1],\right.$ $\left. \mathcal{O}_X, \mathcal{N}_1, \ldots, \mathcal{N}_t \right>$ and of any of its mutations.

\section{Canonical DG algebras of toric surfaces}\label{sec_toric}

\subsection{Toric surfaces}

We recall some information about toric surfaces. More details can be found for example in \cite{bib_F}.

A smooth projective toric surface $Y$ is determined by its fan, spanned by a collection of elements $\rho_1, \ldots, \rho_n $ in a lattice $N = \textrm{Hom}(\mathbb{C}^*, T) \cong \mathbb{Z}^2$, where \mbox{$T = (\mathbb{C}^*)^2$} is a two-dimensional torus. We enumerate $\rho_i$'s clockwise and consider their indexes, $i$'s, to be elements of $\mathbb{Z}/n\mathbb{Z}$. Then, for every $i \in \mathbb{Z}/n \mathbb{Z}$, vectors $\rho_i$ and $\rho_{i+1}$ form an oriented basis of $N$. Moreover, for every such pair there is no other $\rho_k$ lying in the rational polyhedral cone generated by $\rho_i$ and $\rho_{i+1}$ in $N_\mathbb{Q} = N\otimes \mathbb{Q}$.

There is a one-to-one correspondence between one-dimensional orbits of the $T$-action on $Y$ and the rays in the fan generated by $\rho_i$'s. For every $i$ we denote by $D_i$ the closure of this orbit. Then $D_i$'s are $T$-invariant divisors on $X$. Every $D_i$ is isomorphic to $\mathbb{P}^1$ and the intersection form is given by 
\begin{align*}
&D_i D_j =\left\{ \begin{array}{cl}a_i &\textrm{if } i = j,\\ 1 & \textrm{if } j \in \{i-1, i+1\} \\0 &\textrm{otherwise,} \end{array}\right.&
\end{align*}
where $a_i \in\mathbb{Z}$ are such that $\rho_{i-1} + a_i\rho_i + \rho_{i+1} = 0$. Conversely, the numbers $(a_1,\ldots, a_n)$ determine the toric surface $Y$. 

Divisors $D_i$ and $D_{i+1}$ intersect transversely in a $T$-fixed point $p_i$ corresponding to the cone spanned by vectors $\rho_i$ and $\rho_{i+1}$.

A surface $Y_1$ obtained from $Y$ by a blow-up of a torus-fixed point $p_i$ is again a toric surface. The fan of $Y_1$ is determined by vectors $\rho_1, \ldots, \rho_i, \rho_i + \rho_{i+1}, \rho_{i+1}, \ldots, \rho_n$. Moreover, every toric surface different from $\mathbb{P}^2$ can be obtained from some Hirzebruch surface $\mathbb{F}_a$ by a finite sequence of blow-ups of $T$-fixed points.

A canonical divisor of a toric surface is given by $K_Y = - \sum_{i=1}^n D_i$. The Picard group of $Y$ is \mbox{$\textrm{Pic}(Y) = \mathbb{Z}^{n-2}$.}

\subsection{Exceptional collections on toric surfaces}
The $a$-th Hirzebruch surface $\mathbb{F}_a$ has a fan with four vectors and we can assume that $w_1 = (1,0)$, $w_2 = (0,-1)$, \mbox{$w_3 = (-1,a)$} and $w_4=(0,1)$. The collection $\left< \mathcal{O}_{\mathbb{F}_a}, \mathcal{O}_{\mathbb{F}_a}(D_1), \mathcal{O}_{\mathbb{F}_a}(D_1+D_2), \mathcal{O}_{\mathbb{F}_a}(D_1+D_2+D_3) \right>$ is a full strong exceptional collection on $\mathbb{F}_a$. 

If $Y$ is obtained from $\mathbb{F}_a$ by a sequence of $T$-equivariant blow-ups then we can assume that the vectors $\rho_1, \ldots, \rho_n$ determining $Y$ are numbered in such a way that $\rho_n = w_4 = (0,1)$. Then the collection $\left< \mathcal{O}_Y, \mathcal{O}_Y(D_1), \mathcal{O}_Y(D_1+D_2), \ldots, \mathcal{O}_Y(D_1+\ldots+D_{n-1}) \right>$ on $Y$ is obtained by augmentation from the strong collection on $\mathbb{F}_a$ and hence it is full. The following lemma tells us that in fact the numeration of $T$-invariant divisors is not important. 
\begin{lem}[cf. Theorem 4.1 of \cite{bib_B}]\label{lem_mutations}
Let $\left< \mathcal{E}_1, \ldots, \mathcal{E}_n \right>$ be a full exceptional collection on a smooth projective variety $Z$ of dimension $m$. Then the $n$-fold mutation of $\mathcal{E}_n$ to the left, $L^n \mathcal{E}_n = \mathcal{E}_n\otimes \omega_Z[m-n]$, where $\omega_Z$ is the canonical line bundle on $Z$.
\end{lem}
Let $\sigma_1 = \left< \mathcal{O}_Y, \mathcal{O}_Y(D_1), \mathcal{O}_Y(D_1+D_2), \ldots, \mathcal{O}_Y(D_1+\ldots+D_{n-1}) \right>$ be a full exceptional collection on $Y$. Then, by the above lemma 
$$
L^n\, \mathcal{O}_Y(D_1+\ldots+D_{n-1}) = \mathcal{O}_Y(-D_n)[2-n].
$$
Hence, $\sigma_1$ can be mutated to a collection 
$$
\left<\mathcal{O}_Y(-D_{n})[2-n],\mathcal{O}_Y, \mathcal{O}_Y(D_1), \mathcal{O}_Y(D_1+D_2), \ldots, \mathcal{O}_Y(D_1+\ldots+D_{n-2}) \right>
$$
which, in turn, after a shift and a twist by $\mathcal{O}_Y(D_n)$ is equivalent to the collection 
$$
\sigma_n = \left<\mathcal{O}_Y, \mathcal{O}_Y(D_n), \mathcal{O}_Y(D_n+D_1), \ldots, \mathcal{O}_Y(D_n+D_1+\ldots+D_{n-2}) \right>.
$$
One can repeat this operation and obtain full exceptional collections 
$$
\sigma_i = \left< \mathcal{O}_Y, \mathcal{O}_Y(D_i), \ldots, \mathcal{O}_Y(D_i+ \ldots+D_{i+n-2}) \right>
$$
for any $i\in \mathbb{Z}/n$.

\subsection{Canonical DG algebra of a toric surface}

Let $Z = \textrm{Tot}\, \omega_Y$ be the total space of the canonical bundle on $Y$ and let $p\colon Z \rightarrow Y$ denote the canonical projection. As the vector bundle $\mathcal{E} = \mathcal{O}_Y \oplus \mathcal{O}_Y(D_1) \oplus \ldots \oplus \mathcal{O}_Y(D_1+\ldots+D_{n-1})$ is a generator of $D^b(Y)$, we know that $p^*(\mathcal{E})$ is a generator of $D^b(Z)$. Moreover,
\begin{align*}
&\textrm{Hom}_Z(p^* (\mathcal{E}), p^*(\mathcal{E})) = \textrm{Hom}_Y(\mathcal{E}, p_* p^*(\mathcal{E})) = \\
&=\textrm{Hom}_Y(\mathcal{E}, \mathcal{E} \otimes p_*(\mathcal{O}_Z)) = \bigoplus_{n\geq 0} \textrm{Hom}_Y(\mathcal{E}, \mathcal{E} \otimes \mathcal{O}_Y(-n K_Y)).
\end{align*}

On $Y$ we can consider an infinite sequence $(A_k)_{k=0}^{\infty}$ of line bundles 
$$
A_{sn+r} = \mathcal{O}(s K_Y + D_1 + \ldots+D_r), \quad \textrm{for } 0\leq r < n.
$$
Denote by $\mathcal{A}_Y = \bigoplus A_k$ the sum of all elements in this sequence. It is proved in \cite{bib_S} that the DG enhancement of $\textrm{Hom}^\bullet(\mathcal{A}_Y, \mathcal{A}_Y)$ can be calculated via the \v{C}ech enhancement. It follows that the DG enhancement of $\textrm{Hom}_Z(p^*(\mathcal{E}), p^*(\mathcal{E}))$ is the same as the DG enhancement of $\textrm{Hom}_Y(\mathcal{A}_Y, \mathcal{A}_Y)$.

The sequence $\left< \mathcal{O}_Y, \mathcal{O}_Y(D_1), \ldots, \mathcal{O}_Y(D_1+\ldots+D_{n-1}) \right>$ is an augmentation of a strong exceptional collection on a Hirzebruch surface and therefore the methods described in Section \ref{sec_DGquivers} allow to calculate the DG algebra of endomorphisms of $\bigoplus_{k=0}^{n-1} A_k$. Lemma \ref{lem_mutations} guarantees that up to shifts the remaining elements of the sequence $(A_k)$ are obtained by mutations from $A_0, \ldots, A_{n-1}$. Therefore, twisted complexes allow to calculate the DG endomorphism algebra of $\textrm{Hom}(\mathcal{A}_Y, \mathcal{A}_Y)$, \emph{the canonical DG algebra of }$Y$.

The composition provides a natural map
$$
\textrm{Hom}(A_{i_{k-1}},A_{i_k}) \otimes \ldots \otimes \textrm{Hom}(A_{i_1},A_{i_2}) \xrightarrow{\Psi_{i_1,\ldots,i_k}} \textrm{Hom}(A_{i_1},A_{i_k})
$$
and an analogous one for elements of $\textrm{Ext}^1(A_{i_1},A_{i_k})$. If there exists $K \in \mathbb{N}$ such that for any $i,j$ any element of $\textrm{Hom}(A_i,A_j)$ or $\textrm{Ext}^1(A_i,A_j)$ is in the image of some $\Psi_{i_1,\ldots,i_k}$ such that $i_{s+1} - i_s < K$ for all $s\in \{1,\ldots, k-1\}$ then the canonical DG algebra of $Y$ can be presented as a path algebra of a cyclic DG quiver with $K$ vertices.

If one can choose $K$ to be the number $n$ of $T$-invariant divisors of $Y$ then the DG quivers $Q_i$'s of exceptional collections $\sigma_i$ can be read from the canonical DG quiver $Q$ of $Y$
\begin{align*}
(Q_{i})_0 &= (Q_Y)_0,\\
(Q_{i})_1 &= (Q_Y)_1 \setminus \{a \in (Q_Y)_1 \, | \, t(a) > i-1 >h(a)  \}
\end{align*}
and the canonical DG quiver $Q$ is obtained by glueing of the DG quivers $Q_i$.

\begin{rem}
The canonical DG algebra of $\mathbb{F}_3$ cannot be presented as a path algebra of such a quiver, i.e. in this case $K > 4$. If, as before, we consider the fan of $\mathbb{F}_3$ with $w_1=(1,0)$, $w_2 = (0, -1)$, $w_3 = (-1, 3)$ and $w_4 = (0,1)$ then the map $\phi: \mathcal{O}_{\mathbb{F}_3}(D_1+D_2)\rightarrow \mathcal{O}_{\mathbb{F}_3}(2D_1+2D_2+2D_3+D_4)$ with zeroes along $2 D_2$ is not  a non-trivial composition of any maps between line bundles.
\end{rem}

\subsection{Examples} We conclude with some examples of canonical DG quivers of toric surfaces.

The canonical DG algebra of $\mathbb{F}_1$ is a path algebra of the quiver

\[
\xymatrix{& *+[F]{v_1} \ar@<1ex>[dr]|{a_1} \ar@<-1ex>[dr]|{a_2}& \\
*+[F]{v_4} \ar[ur]|{e} \ar@/^4pc/[rr]|{g} & & *+[F]{v_2} \ar@<2ex>[dl]|{b} \ar[dl]|{c_0} \ar@<-2ex>[dl]|{c_1} \\ 
& *+[F]{v_3} \ar@<1ex>[ul]|{d_1} \ar@<-1ex>[ul]|{d_2} \ar@/^1pc/[uu]|{f}  &}
\]
with relations
\begin{align*}
&c_0\, a_1 = c_1\, a_2, & & d_1 \,c_0 = d_2\, c_1, & &d_1\, b \,a_2 = d_2\, b \,a_1, &\\
&a_1 \,e \,d_2 = a_2 \,e\, d_1, & &a_1 \,f = g\, d_1,& &a_2 \,f = g\, d_2, &\\
&b \,a_1\, e = c_1 \,g, & &b\, a_2\, e = c_0 \,g, & & f\, c_0 = e \,d_2\, b,&\\
&f \,c_1 = e\, d_1\, b.& 
\end{align*}

The canonical DG algebra of $\mathbb{F}_2$, with intersection numbers $(0,2,0,-2)$, is a path algebra of the following DG quiver
\[
\xymatrix{& & *+[F]{v_4} \ar@<1ex>[ddrr]|{a_1} \ar@<-1ex>[ddrr]|{a_2} & & \\
\\
*+[F]{v_3} \ar@<1ex>[uurr]|{e} \ar@<-1ex>[uurr]|{f} \ar@<1ex>@/^1pc/[rrrr]|{h_1} \ar@<-1ex>@/^1pc/[rrrr]|{h_2} & & & & *+[F]{v_1} \ar@<3ex>[ddll]|{b} \ar@<1ex>[ddll]|{c_0} \ar@<-1ex>[ddll]|{c_1} \ar@<-3ex>[ddll]|{c_2} \ar@<1ex>@/^1pc/[llll]|{g_1} \ar@<-1ex>@/^1pc/[llll]|{g_2} \\
\\
& & *+[F]{v_2} \ar@<1ex>[uull]|{d_1}\ar@<-1ex>[uull]|{d_2} \ar@<2ex>@/^8pc/[uuuu]|{j_2} \ar@/^8pc/[uuuu]|{j_1} & &}
\]
with
\begin{align*}
&\textrm{deg}(a_1) = 0,& &\textrm{deg}(a_2) = 0,& &\textrm{deg}(b) = 0,& &\textrm{deg}(c_0) = 0,&\\
&\textrm{deg}(c_1) = 0,& &\textrm{deg}(c_2) = 0,& &\textrm{deg}(d_1) = 0,& &\textrm{deg}(d_2) = 0,&\\
&\textrm{deg}(e) = 0,& &\textrm{deg}(f) =1,& &\textrm{deg}(g_1) = -1,& &\textrm{deg}(g_2) = -1,&\\
&\textrm{deg}(h_1) = 0,& &\textrm{deg}(h_2)=0,& &\textrm{deg}(j_1) = 0,& &\textrm{deg}(j_2) = 0,&
\end{align*}
\begin{align*}
&\partial(g_1)=d_2\,c_1\,-\,d_1\,c_0,& &\partial(g_2)=d_2\,c_2\,-\,d_1\,c_1,& &\partial(h_1)=a_1\,f,&\\
&\partial(h_2)=a_2\,f,& &\partial(j_1)=f\,d_1,& &\partial(j_2)=f\,d_2&
\end{align*}
and relations
\begin{align*}
&c_0\, a_1 = c_1\, a_2,& &c_1\,a_1 = c_2\, a_2,& &d_1\, b \, a_2 = d_2\, b\, a_1,& &c_1\, h_2 = c_0\, h_1 + b\, a_2 \, e,&\\
&c_2\, h_2 = c_1\,h_1 + b\, a_1\,e,& &a_1\,j_2 = a_2\, j_1,& &h_1\, d_2 = h_2\, d_1,& &a_1\,e \, d_2 = a_2\, e \, d_1,& \\
&a_1\, f\, d_2 = a_2\, f\, d_1, & & f\,d_1\, c_0 = f\, d_2\, c_1,& & f\,d_1\,c_1 = f\, d_2\, c_2,& &f\,g_1 = e\,d_2\, b,&\\
&f\, g_2 = e\, d_1\, b,& & j_1\,c_0 = j_2\, c_1,& &j_1\, c_1 = j_2 \,c_1,& & a_1\, j_1 = 0,&\\
&a_2 \,j_2 = 0,& & h_1\,d_1 =0,& &h_2\, d_2 = 0.&  
\end{align*}
If we blow up $\mathbb{F}_1$ in such a way that the obtained toric surface $Y_1$ has intersection numbers $(-1,-1,0,0,-1)$ then the canonical algebra of $Y_1$ is a path algebra of the following quiver
\[
\xymatrix{&& *+[F]{v_5} \ar@<1ex>@/^1pc/[drr]|{a} \ar@<-1ex>@/^1pc/[drr]|{b}  && \\
*+[F]{v_4} \ar@/^1pc/[urr]|{j} \ar@/^1pc/[rrrr]|{l}& & & & *+[F]{v_1} \ar@<1ex>[d]|{c} \ar@<-1ex>[d]|{d} \ar@/^1pc/[dllll]|(0.3){g} \\
*+[F]{v_3} \ar[u]|{f} \ar@/^1pc/[uurr]|(0.4){k} & & & & *+[F]{v_2}\ar@/^1pc/[llll]|{e} \ar@/^1pc/[ullll]|(0.7){h} }
\]
with relations
\begin{align*}
& g\, b = e\, d\, a, & &h\, d = f\, g,& &h\, c\, b = f\, e\, c\, a, & &k\, g = j\, h\, c,& \\
&k \, e \, d = j\, f\, e\, c,& &b\, k = l\, f,& &b\, j \, h = e\, j\, f\, e, & &l\, h = a\, k\, e, & \\
&l\, f\, e = b\, k\, e,  & &d\, l = c\, b\, j, & &b\, k = l\, f,&&d\, a\, k = c\, a\, j\, f,& \\
&g\, l = e\, c\, a\, j.&
\end{align*}
If we blow up $\mathbb{F}_1$ in another point, to obtain $Y_2$ with intersection numbers $(0,1,-1,-1,-2)$, then the canonical DG algebra is a path algebra of the following DG quiver: 

\[
\xymatrix{*+[F]{v_4} \ar[rrr]|{g}  \ar@/^5pc/[rrrrrd]|{m}  \ar@/^2pc/[rrrdd]|{r} & & & *+[F]{v_5} \ar@<-1ex>@/^2pc/[dd]|{l_1}  \ar@<1ex>@/^2pc/[dd]|{l_2} \ar@<1ex>@/^1pc/[drr]|{k_1}   \ar@<-1ex>@/^1pc/[drr]|{k_2} & & \\
& & & & & *+[F]{v_1} \ar@<1ex>@/^1pc/[dll]|{a} \ar@<-1ex>@/^1pc/[dll]|{b} \\
*+[F]{v_3}\ar[uu]|{f} \ar@/^2pc/[uurrr]|{h} & & & *+[F]{v_2} \ar@<-1ex>[uu]|{s_1} \ar@<1ex>[uu]|{s_2} \ar@<2ex>@/^1pc/[lll]|{c} \ar@/^1pc/[lll]|{d} \ar@<-2ex>@/^1pc/[lll]|{e} \ar@/^2pc/[uulll]|{i}& &}
\]
with
\begin{align*}
&\textrm{deg}(a) = 0,& &\textrm{deg}(b)=0,& &\textrm{deg}(c)=0,& &\textrm{deg}(d)=0,&\\
&\textrm{deg}(e)=0,& &\textrm{deg}(f)=0,& &\textrm{deg}(g)=0,& &\textrm{deg}(h)=0,&\\
&\textrm{deg}(i)=0,& &\textrm{deg}(k_1)=1,& &\textrm{deg}(k_2)=0,& &\textrm{deg}(l_1)=0,&\\
&\textrm{deg}(l_2)=0,& &\textrm{deg}(m)=0,& &\textrm{deg}(r)=0,& &\textrm{deg}(s_1)=-1,&\\
&\textrm{deg}(s_2)=-1,&
\end{align*}
\begin{align*}
&\partial(l_1)=b\,k_1,& &\partial(l_2)=b\,k_2,& &\partial(m)=k_1\,g,& \\
&\partial(r)=k_1\, h,&&\partial(s_1)=h\, e \, - \, g\, i,& &\partial(s_2)=h\,d \, - \, g \, f\, e&
\end{align*}
and relations
\begin{align*}
&e\, b = d\, a,& &i\, b = f\, e\, a,& &g\, f\, c\, a= h\, c\, b,& &e\, l_1 = c\, b\, k_2+d\, l_2,&\\
 &i\, l_1 = f\, e\, l_2 + f\, c\, a \, k_2,& &l_1\, g = b\, m,& &l_2\, g = a\, m,& &a\, r = l_2\, h,&\\
 &a\, m\, f = b\, r,& &b\, k_2\, h = a \, k_2\, g\, f,& &b\, k_1\, h = a\, k_1\, g\, f,& &k_1\, s_1 = k_2\, h\, c,&\\
 &k_1\, s_2 = k_2\, g\, f\, c.&
\end{align*}

\textbf{Acknowledgements}
I would like to thank Alexey Bondal, Alexander Efimov, Alexander Kuznetsov and Markus Perling for useful discussions and suggestions.

\end{document}